\documentclass[12pt,a4paper]{article}
\usepackage{amssymb,amsmath,amsthm}
\usepackage{graphicx}
\usepackage{hyperref}

\newcommand\cA{{\mathcal A}}
\newcommand\cB{{\mathcal B}}
\newcommand\cC{{\mathcal C}}

\newcommand\cF{{\mathcal F}}

\newcommand\ex{\ensuremath{\mathrm{ex}}}
\newcommand\cex{\ensuremath{\mathrm{ex}^{\mathrm{col}}}}

\newcommand\cG{{\mathcal G}}
\newcommand\cH{{\mathcal H}}

\newcommand\cN{{\mathcal N}}
\newcommand\cP{{\mathcal P}}

\newcommand\cT{{\mathcal T}}

\theoremstyle{plain}
\newtheorem{theorem}{Theorem}[section]
\newtheorem{lemma}[theorem]{Lemma}
\newtheorem{corollary}[theorem]{Corollary}
\newtheorem{conjecture}[theorem]{Conjecture}
\newtheorem{proposition}[theorem]{Proposition}

\theoremstyle{definition}

\newtheorem{claim}[theorem]{Claim}

\newtheorem*{lemma*}{Lemma}
\newtheorem*{thm*}{Theorem}

\newcommand\cref[1]{Corollary~\ref{cor:#1}}

\title{Counting multiple graphs in generalized Tur\'an problems}

\author{D\'aniel Gerbner}

\begin{document}

\maketitle

\begin{abstract}
    We are given graphs $H_1,\dots,H_k$ and $F$. Consider an $F$-free graph $G$ on $n$ vertices. What is the largest sum of the number of copies of $H_i$? The case $k=1$ has attracted a lot of attention.
 We also consider a colored variant, where the edges of $G$ are colored with $k$ colors. What is the largest sum of the number of copies of $H_i$ in color $i$? 
    
    Our motivation to study this colored variant is a recent result stating that the Tur\'an number of the $r$-uniform Berge-$F$ hypergraphs is at most the quantity defined above for $k=2$, $H_1=K_r$ and $H_2=K_2$. 
    
    In addition to studying these new questions, we obtain  new results for generalized Tur\'an problems and also for Berge hypergraphs.
\end{abstract}

\section{Introduction}

For graph $H$ and $G$, let $\cN(H,G)$ denote the number of subgraphs of $G$ that are isomorphic to $H$. Let $\ex(n,H,F)$ denote the largest $\cN(H,G)$ among $F$-free graphs $G$ on $n$ vertices. In the case $H=K_2$, the Tur\'an number $\ex(n,F):=\ex(n,K_2,F)$ is one of the most studied parameters in extremal graph theory. The systematic study of the general version has been initiated by Alon and Shikhelman \cite{as}, after several sporadic results.

In this paper we study a variant, where instead of counting copies of a subgraph $H$, we count copies of several different subgraphs.
Let $H_1,\dots, H_k$ and $G$ be graphs.  Let $\cN(H_1,\dots, H_k;G)=\sum_{i=1}^k\cN(H_i,G)$ and let $\ex(n,(H_1,\dots, H_k),F)$ denote the largest value of $\cN(H_1,\dots, H_k;G)$ if $G$ is an $F$-free graph on $n$ vertices. 

The first thing to observe is that if for every $i$ the same graph $G$ maximizes $\cN(H_i,G)$ among $n$-vertex $F$-free graphs, then we are done, the sum is also maximized by that graph. This is the case for cliques. The Tur\'an graph $T_r(n)$ is a complete $r$-partite graph where each part has size $\lfloor n/r\rfloor$ or $\lceil n/r\rceil$. Tur\'an \cite{T} showed that $\ex(n,K_k)=|E(T_{k-1}(n))|$ and Zykov \cite{zykov} showed that $\ex(n,K_r,K_{\ell})=\cN(K_r,T_{\ell-1}(n))$. This implies $\ex(n,(H_1,\dots, H_k),K_\ell)=\cN(H_1,\dots, H_k;T_{\ell-1}(n))$ in the case each $H_i$ is a clique.

One could also consider a weighted version, where we are also given $\alpha_1,\dots,\alpha_k$ and we want to maximize $\sum_{i=1}^k\alpha_i\cN(H_i,G)$. This was studied by Bollob\'as \cite{boll}, who showed that in the case each $H_i$ is a clique, then a complete multipartite graph gives the maximum, but not necessarily the Tur\'an graph (note that $\alpha_i<0$ is possible). He proved it for $k=2$, but the proof easily extends for larger $k$. Schelp and Thomason \cite{scth} extended it to induced copies of complete multipartite graphs $H_i$, in the case for every $i$ either $\alpha_i\ge 0$, or $H_i$ is a complete graph. They remark that we can assume that the clique number is bounded, i.e. some $K_r$ is forbidden. Note that we deal with not necessarily induced copies of $H_i$. However, all complete $(\ell-1)$-partite subgraphs and all cliques are induced in a $K_\ell$-free graph, thus we have that if each $H_i$ is a complete $(\ell-1)$-partite graph or a clique, then $\ex(n,(H_1,\dots, H_k),K_\ell)=\cN(H_1,\dots, H_k;G)$ for some complete $(\ell-1)$-partite graph $G$. For any particular integer $\ell$ and graphs $H_i$, a straightforward optimization would find the extremal graph, but we cannot handle it in this generality.

Other results that fit into this setting are when we count structures that correspond to multiple subgraphs, for example walks. Another example is the second Zagreb index of a graph $\sum_{uv\in E(G)} d(u)d(v)$ (see \cite{bdfg} for a survey), which is equal to $\cN(P_4,G)+3\cN(K_3,G)+\cN(P_3,G)+2\cN(K_2,G)$. We denote by $P_\ell$ the path on $\ell$ vertices and by $S_\ell$ the star on $\ell$ vertices. To observe this equality, we pick an edge $uv$ and pick one of the neighbors $u'$ of $u$ and one of the neighbors $v'$ of $v$. If $u$ and $u'$ are distinct from each other and from $u$ and $v$, we obtain a $P_4$. If $u'=v'$, we obtain a triangle, and each edge is counted three times this way. If $u'=v$ and $v'=u$, we obtain an edge, and each edge is counted twice this way. In the remaining case either $u'=v$ and $v'\neq u$ or $v'=u$ and $u'\neq v$, thus we obtain two copies of $P_3$, but each $P_3$ is counted twice altogether. 

Another example is the problem of studying $\sum_{i=1}^n d_i^r$ in $F$-free graphs with degree sequence $d_1,\dots, d_n$. This was initiated by Caro and Yuster \cite{cy} and shown to be equal to a weighted sum of the number of stars in \cite{ge5}. Another motivation comes from \cite{gmv,ge4} where for some graphs $H$ and $F$, bounds on $\ex(n,H,F)$ were given by using a (weighted) sum of the number of copies of subgraphs of $H$. 

In what follows, we do not deal with weights (with the exception of one remark). Many of our proofs immediately extend to a weighted version, but we feel the most important effect of adding weights would be the even more complicated notation. Similarly, we avoid forbidding multiple graphs at the same time, just for the sake of simplicity.

Let us mention that there are results on forbidding a subgraph and counting multiple subgraphs where all graphs belonging to an infinite family are counted, for example all the cycles, see \cite{mrs} and the references in it.

\smallskip

We also consider a colored variant. Our main motivation to study this variant is its connection to Berge hypergraphs, that we will describe later.
Let $G$ be a graph with edges colored by $1,\dots,k$. Then we denote by $G_i$ the subgraph of $G$ having the edges of color $i$. Let $\cN^{{\mathrm col}}(H_1,\dots, H_k;G)=\sum_{i=1}^k\cN(H_i,G_i)$.
Let $\cex(n,(H_1,\dots, H_k),F)$ denote the largest $\cN^{{\mathrm col}}(H_1,\dots, H_k;G)$ if $G$ is a $k$-colored $F$-free graph on $n$ vertices. In the case $k=2$, we call the first color blue and the second color red. 

We will refer to $\cex(n,(H_1,\dots, H_k),F)$ as the colored variant, and $\ex(n,(H_1,\dots, H_k),F)$ as the uncolored variant.
Let us start with some simple observations.
\begin{proposition}\label{obse} For any $i$ we have $$\ex(n,H_i,F)\le \cex(n,(H_1,\dots, H_k),F)\le\ex(n,(H_1,\dots, H_k),F)\le\sum_{i=1}^k \ex(n,H_i,F).$$

\end{proposition}

\begin{corollary} We have $\ex(n,(H_1,\dots, H_k),F)=\Theta(\max_{i\le k} \ex(n,H_i,F))$ and similarly $\cex(n,(H_1,\dots, H_k),F)=\Theta(\max_{i\le k} \ex(n,H_i,F))$.

\end{corollary}

Thus, if we know $\ex(n,H_i,F)$ for every $i$, then we know the order of magnitude of $\ex(n,(H_1,\dots, H_k),F)$ and $\cex(n,(H_1,\dots, H_k),F)$. We conjecture that in the colored variant, the same lower bound is asymptotically sharp.

\begin{conjecture}\label{mainconj}
$$\cex(n,(H_1,\dots, H_k),F)=(1+o(1))\max_{i\le k} \ex(n,H_i,F).$$
\end{conjecture}

Observe that if the order of magnitude of $\ex(n,H_j,F)$ is larger than the order of magnitude of $\ex(n,H_{\ell},F)$ for any $\ell\neq j$, then the asymptotic result immediately follows using Proposition \ref{obse}, even in the uncolored variant. We will see an example in Section 4 showing that the analogue of Conjecture~\ref{mainconj} does not hold in the uncolored case.

In light of the above observations, the interesting results are asymptotic ones in the few cases they are non-trivial, and exact results.
We say that the $(k+1)$-tuple of graphs $(H_1,\dots,H_k,F)$ is \textit{color-resistant} for an integer $n$ if $\cex(n,(H_1,\dots, H_k),F)=\max_{i\le k} \ex(n,H_i,F)$, i.e. a monochromatic graph attains the maximum.

Note that $\ex(n,(H_1,\dots, H_k),F)$ can be equal to $\max_{i\le k} \ex(n,H_i,F)$ only if the extremal $n$-vertex $F$-free graph $G$ contains no copies of $H_j$ with $j\neq i$ (we will see several examples of this later).
What happens more often is that $G$ also gives the maximum for our problem, i.e., $\max_{i\le k} \ex(n,H_i,F)=\ex(n,H_j,F)=\cN(H_j,G)$ and also  $\ex(n,(H_1,\dots, H_k),F)=\cN(H_1,\dots,H_k;G)$. In this case we say that the $(k+1)$-tuple of graphs $(H_1,\dots,H_k,F)$ is \textit{resistant}.

\smallskip

Our main motivation to study the colored variant is its application in the theory of Berge hypergraphs. We say that a hypergraph $\cH$ is a Berge copy of a graph $F$ (in short: $\cH$ is a Berge-$F$) if $V(F)\subset V(\cH)$ and there is a bijection $f:E(F)\rightarrow E(\cH)$ such that for any $e\in E(F)$ we have $e\subset f(e)$. This definition was introduced by Gerbner and Palmer \cite{gp1}, extending the well-established notion of Berge cycles and paths. The largest number of hyperedges in an $r$-uniform Berge-$F$-free hypergraph is denoted by $\ex_r(n,\textup{Berge-}F)$, see Chapter 5.2.2 of \cite{gp} for a short survey on this function.

Gerbner and Palmer \cite{gp2} connected Berge hypergraphs and generalized Tur\'an problems by showing $\ex(n,K_r,F)\le \ex_r(n,\textup{Berge-}F)\le \ex(n,K_r,F)+\ex(n,F)$. This was improved by F\"uredi, Kostochka and Luo \cite{fkl} and independently by Gerbner, Methuku and Palmer \cite{gmp} to $\ex_r(n,\textup{Berge-}F)\le \cex(n,(K_r,K_2),F)$.

Let us briefly state their results in our setting. F\"uredi, Kostochka and Luo \cite{fkl} gave an upper bound on $\cex(n,(K_r,K_2),\cC_{\ell})$, where $\cC_\ell$ denotes the set of cycles of length at least $\ell$. This upper bound is sharp in the case $\ell-2$ divides $n-1$ and $\ell\ge r+3$. In \cite{fkl2} they determined $\cex(n,(K_r,K_2),\cC_{\ell})$ for every $n$ if $\ell\ge r+4$. Gerbner, Methuku and Palmer \cite{gmp} showed $(K_r,K_2,K_m)$ is color-resistant for any $n$, $r$ and $m$. They also showed that for any graph $F$, if $F'$ is obtained by deleting a vertex of $F$, and $f(n)$ is such that $\ex(n,K_{r-1},F')\le f(n)n$ for every $n$, then $\cex(n,(K_r,K_2),F)\le \max\{2f(n)/r,1\}\ex(n,F)$. With this, they gave bounds on the Tur\'an number of several different Berge hypergraphs; those results also extend to our setting. 

In particular the proofs in \cite{gmp} imply $\cex(n,(K_r,K_2),T)\le \frac{n}{k}\binom{k}{r}$ for every tree $T$ on $k+1$ vertices if $k>r+1>3$, assuming the Erd\H os-S\'os conjecture on the Tur\'an number of trees holds for every subtree of $T$. Note that this is known for some classes of trees, for example if $T$ is a path \cite{eg}, spider \cite{fhl} or has diameter at most four \cite{mcl}. Also $\cex(n,(K_3,K_2),K_{2,t})=(1+o(1))(t-1)^{3/2}n^{3/2}/6$ and $\cex(n,(K_3,K_2),C_{2k})\le (2k-3)\ex(n,C_{2k})/3$ is implied for $t\ge 4$ and $k\ge3$ from the proofs in \cite{gmp}, along with other bounds.

\smallskip

Let us describe the structure of the paper. We also state here our new results concerning generalized Tur\'an problems and Berge hypergraphs, as we believe they may be of interest to more readers.
In Section 2, we describe how stability results concerning generalized Tur\'an numbers can be used in our setting. We apply the few known results and prove a new one. 

We say that an edge $uv$ of a graph $G$ is a color-critical edge if deleting it from $G$ decreases its chromatic number. An $m$-chromatic graph $F$ with a color-critical edge often behaves similarly to $K_m$ in extremal problems. In particular, Simonovits \cite{miki} showed that for $n$ large enough, the Tur\'an graph $T_{m-1}(n)$ contains the most edges among $F$-free graphs, and it was extended by Ma and Qiu \cite{mq}, who showed that $T_{m-1}(n)$ also contains the most copies of $K_r$ for $r<m$. We prove that there is also stability here.

\begin{lemma}\label{stabi} Let $F$ be an $(m+1)$-chromatic graph with a color-critical edge and $r<m+1$. If $G$ is an $n$-vertex $F$-free graph with chromatic number more than $m$, then $\ex(n,K_r,F)-\cN(K_r,G)=\Omega(n^{r-1})$.
\end{lemma}

For Berge hypergraphs, we obtain the following. 

\begin{proposition}\label{bergecoro} Let $\chi(F)>r$. Then $\ex_r(n,\textup{Berge-}F)=\ex(n,K_r,F)+o(n^2)$.
\end{proposition}

We remark that this bound was improved to $\ex_r(n,\textup{Berge-}F)=\ex(n,K_r,F)+O(1)$ in \cite{ge3}.
In Section 3, we deal with the colored variant. We show that each tuple of cliques is color-resistant for every $n$, and show some tuples that are not color-resistant.

In Section 4 we deal with the uncolored variant. We show that the analogue of Conjecture \ref{mainconj} does not hold, and also examine some particular instances of the problem, where we count some graphs on five vertices in triangle-free graphs.

\section{Proofs using stability}

Let us describe an approach to show that a tuple $(H_1,\dots,H_k,F)$ is resistant and/or color resistant for $n$ large enough. 
Assume without loss of generality that $\max_{i\le k}\ex(n,H_i,F)=\ex(n,H_1,F)=\cN(H_1,G_0)$, where $G_0$ is $F$-free and has $n$ vertices. Moreover, assume $\ex(n,H_1,F)$ is way larger than $\ex(n,H_i,F)$ for every $i>1$. Assume furthermore that there is a stability result concerning $\ex(n,H_1,F)$, stating that if $G$ is $F$-free on $n$ vertices, and has at least $\ex(n,H_1,F)-x$ copies of $H_1$, then $G$ is a subgraph of $G_0$. 
Now if $x>\sum_{i=2}^k\ex(n,H_i,F)$, then the graph which maximizes $\cN(H_1,\dots,H_k;G')$ has to be a subgraph of $G_0$, hence we can assume it is $G_0$. Thus we solved the uncolored variant, and in the colored variant we reduced the problem to show that the best coloring of $G_0$ is when every edge has color 1.

In the first version of this manuscript in 2020 I wrote that unfortunately, there are not many stability results for generalized Tur\'an problems. However, this has changed since then, see e.g. \cite{dahi} and the citations within. Those results can be used to obtain further bounds and exact results in our settings.





Gerbner and Palmer \cite{gp3} showed that for $n$ large enough, we have $\ex(n,P_4,C_5)=\cN(P_4,T_2(n))$. Moreover, if a $C_5$-free graph $G$ has $\alpha$ edges that are contained in triangles, then $\cN(P_4,G)\le \cN(P_4,T_2(n))-(1+o(1))\alpha n^2/12$. This result shows that a $C_5$-free graph which has almost the largest possible number of $P_4$s must be close to $T_2(n)$, but it measures the ``distance'' from $T_2(n)$ in an unusual way. For us, this is a very useful way though, as it implies the following.

\begin{proposition}\label{colres}
$(P_4,K_3,C_5)$ is resistant and color-resistant for every $n$ large enough. 
\end{proposition}
We will also use a result of Bollob\'as and Gy\H ori \cite{BGy} that states $\ex(n,K_3,C_5)=O(n^{3/2})$.

\begin{proof} Let $G$ be a $C_5$-free graph.
If there is a triangle in $G$, then $\cN(K_3,G)+\cN(P_4,G)\le \ex(n,K_3,C_5)+\ex(n,P_4,C_5)-(1+o(1))\alpha n^2/12< \ex(n,P_4,C_5)$ for $n$ large enough.
\end{proof}

Let us return to the weighted variant for an observation. The above proof shows that if we add up the number of $P_4$'s plus three times the number of triangles in $G$, we obtain the same upper bound. For sufficiently large $n$, a theorem of Simonovits \cite{miki} implies that $\ex(n,C_5)=|E(T_2(n))|$ while a result of Gerbner \cite{gerbner2} implies that $\ex(n,P_3,C_5)=\cN(P_3,(T_2(n))$. Therefore, for $n$ large enough, among $n$-vertex $C_5$-free graphs $T_2(n)$ has the largest second Zagreb index. 
Another stability result is due to Ma and Qiu \cite{mq}, who showed the following.

\begin{lemma}[Ma, Qiu \cite{mq}]\label{stabimaqi} Let $F$ be a graph with $\chi(F)=m+1>r\ge 2$. If $G$ is an $n$-vertex $F$-free
graph with $\cN(K_r,G) \ge \cN(K_r,T_m(n)) - o(n^r)$, then $G$ can be obtained from $T_m(n)$ by
adding and deleting a set of $o(n^2)$ edges.

\end{lemma}

\begin{corollary} Let $\chi(F)=m+1>r$, $H_1=K_r$ and for every $i>1$, $|V(H_i)|\le p<r$. Then  $\cex(n,(H_1,\dots,H_k),F)=\ex(n,H_1,F)+o(n^p)$. 

\end{corollary}

\begin{proof}
Let $G$ be an $F$-free graph with $\cN^{\mathrm{col}}(H_1,\dots,H_k;G)=\cex(n,(H_1,\dots,H_k),F)$. If $\cN(K_r,T_{m}(n))-\cN(K_r,G)=\Omega(n^r)$, then $\cN^{\mathrm{col}}(H_1,\dots,H_k;G)<\cN(K_r,T_{m}(n))$ (a contradiction), as there are $O(n^{r-1})$ copies of $H_2,\dots,H_k$ in $G$. Otherwise we can apply Lemma \ref{stabimaqi}, thus we can obtain $G$ by adding and deleting a set of $o(n^2)$ edges from $T_{\chi(F)-1}(n)$. 

Let $G'$ be the common part of $G$ and $T_{m}(n)$ (thus it can be obtained from either $G$ or $T_m(n)$ by deleting $o(n^2)$ edges). If $|E(T_{m})|-|E(G'_1)|=\Omega(n^2)$, then $G_1'$ has $\cN(K_r,T_{m}(n))-\Omega(n^r)$ copies of $K_r$. Every other copy of $H_1$ in $G$ contains at least one of the $o(n^2)$ edges added to $T_{\chi(F)-1}(n)$, thus there are $o(n^r)$ of them. Again, there are $O(n^{r-1})$ copies of $H_2,\dots,H_k$ in $G$, thus $\cN^{\mathrm{col}}(H_1,\dots,H_k;G)<\cN(K_r,T_{m}(n))$, a contradiction.

Therefore, there are $o(n^2)$ edges in $G'$ that are not of color 1, and there are $o(n^2)$ edges of $G$ not in $G'$. Therefore, there are $o(n^2)$ edges of each other color, thus there are $o(n^{|V(H_i)|})=o(n^p)$ copies of $H_i$ in $G_i$, for every $i>1$. There are at most $\ex(n,H_1,F)$ copies of $H_1$ in $G_1$, thus we are done.
\end{proof}

Recall that $ex_r(n,\textup{Berge-}F)\le \cex(n, (K_r,K_2),F)$
(a result of \cite{gp1}, which we mentioned in Section 1).
This implies Proposition \ref{bergecoro}.
However, Lemma \ref{stabimaqi} is not strong enough to obtain a sharp result in our setting. Therefore, we prove Lemma \ref{stabi} for graphs with a color-critical edge. 
We restate Lemma \ref{stabi} here for convenience.

\begin{lemma*} Let $F$ be an $(m+1)$-chromatic graph with a color-critical edge and $r<m+1$. If $G$ is an $n$-vertex $F$-free graph with chromatic number more than $m$, then $\ex(n,K_r,F)-\cN(K_r,G)=\Omega(n^{r-1})$.
\end{lemma*}

This follows easily from a result of
Erd\H os and Simonovits \cite{ersim}. They, extending a result of Andr\'asfai, Erd\H os and S\'os \cite{aes}, showed that if $F$ has a color-critical edge and is $(m+1)$-chromatic, and $G$ is an $F$-free graph on $n$ vertices with chromatic number greater than $m$, then $G$ has a vertex of degree at most $(1-\frac{1}{m-1/3})n$. We also use the following result of Alon and Shikhelman \cite{as}: if $\chi(H)=t>s$, then $\ex(n,K_s,H)=(1+o(1))\binom{t-1}{s}\left(\frac{n}{t-1}\right)^s$.

\begin{proof} By the above, $G$ has a vertex $x$ of degree at most $d=(1-\frac{1}{m-1/3})n$. Let $uv$ be an edge of $F$ whose deletion decreases the chromatic number, and let $F'$ be the graph we obtain from $F$ by deleting $v$. Then $\chi(F')=m$. The neighborhood of $x$ is obviously $F'$-free, thus contains at most $(1+o(1))\cN(K_{r-1},T_{m-1}(d))$ copies of $K_{r-1}$ by the result of Alon and Shikhelman mentioned before the proof. Therefore, in $G$ the number of copies of $K_r$ containing $x$ is at most $(1+o(1))\cN(K_{r-1},T_{m-1}(d))$, while the number of copies of $K_r$ not containing $x$ is at most $\cN(K_r,T_{m}(n-1))$.
Let $y$ be a vertex in a largest class of the Tur\'an graph $T_{m}(n)$. Then the number of copies of $K_r$ containing $y$ is $(1+o(1))\cN(K_{r-1},T_{m-1}(\lfloor (1-\frac{1}{m})n\rfloor)$ , while the number of copies of $K_r$ not containing $y$ is $\cN(K_r,T_{m}(n-1))$. The difference is obviously $\Omega(n^{r-1})$, finishing the proof.
\end{proof}

Note that the bound $\Omega(n^{r-1})$ is sharp, at least for $F=K_{m+1}$, as shown by the following example. We take $T_{m}(n)$, 
and take vertices $a,a'$ in part $A$ and $b$ in part $B$. We delete the edges between $a$ and vertices of $B$, except we keep $ab$, and then add the edge $aa'$. It is easy to see that the resulting graph $G$ is $K_{m+1}$-free and its chromatic number is $m+1$. Compared to the Tur\'an graph, every $K_r$ that got deleted contains $a$, thus there are $O(n^{r-1})$ of them.

In the case $r=2$, Simonovits \cite{miki} gave stronger bounds on the smallest possible value of $\ex(n,K_r,F)-\cN(K_r,G)$: he showed it is between $n/m+c_1$ and $n/m+c_2$ for some constants $c_1$ and $c_2$. In the case $r=2$ and $F=K_{k+1}$, Brouwer \cite{br} determined the above difference exactly.
It would be of interest to obtain a stronger bound than Lemma \ref{stabi}. Still, it is enough for us to obtain the following.

\begin{proposition}\label{edgecrit}
Let $\chi(F)=m+1>r$ and assume $F$ has a color-critical edge, $H_1=K_r$ and for every $i>1$, $|V(H_i)|<r-2$. If $n$ is large enough, then $\cex(n,(H_1,\dots,H_k),F)=\ex(n,H_1,F)=\cN(K_r,T_{m}(n))$ and $\ex(n,(H_1,\dots,H_k),F)= \cN(H_1,\dots,H_k;T_{m}(n))$.
\end{proposition}



\begin{proof} Let $G$ be an $F$-free graph. If $G$ has chromatic number more than $\chi(F)-1$, then by Lemma \ref{stabi} we have $\cN(K_r,G)\le \cN(K_r,T_{\chi(F)-1}(n))-\Omega(n^{r-1})$. As there are $O(n^{r-2})$ copies of the other graphs $H_i$, we are done.

If $G$ has chromatic number at most $\chi(F)-1$, then we can assume it is a complete $(\chi(F)-1)$-partite graph. A simple calculation shows that if $G$ is not the Tur\'an graph, then we have $\cN(K_r,G)\le \cN(K_r,T_{\chi(F)-1}(n))-\Omega(n^{r-2})$. 
As there are $O(n^{r-3})$ copies of the other graphs $H_i$, we are done.
\end{proof}

Let $\cT_m^r(n)$ denote the hypergraph having the $r$-cliques of $T_m(n)$ as hyperedges.

\begin{corollary}\label{criti} Let $\chi(F)=m+1>r\ge 5$ and $F$ have a color-critical edge. If $n$ is large enough, then $\ex_r(n,\textup{Berge-}F)=|\cT_{m}^r(n)|$.
\end{corollary}

Note that the above corollary is already known, for every $r$. The $r$-uniform \textit{expansion} $F^{+r}$ of a graph $F$ is the specific $r$-uniform Berge copy that contains the most vertices, i.e. the $r-2$ vertices added to each edge of $F$ are distinct for different edges, and distinct from the vertices of $F$. Let $\cT_{m}^r(n)$ denote the complete $m$-partite $r$-uniform hypergraph on $n$ vertices with each part of order $\lfloor n/m\rfloor$ or $\lceil n/m\rceil$. Mubayi \cite{muba} proved $\ex_r(n,K_{m+1}^{+r})=(1+o(1))|\cT_{m}^r(n)|$, and Pikhurko \cite{pikhu} proved $\ex_r(n,K_{m+1}^{+r})=|\cT_{m}^r(n)|$ for $n$ large enough.
According to the survey \cite{mubver} of Mubayi and Verstra\"ete on expansions, Alon and Pikhurko observed that Pikhurko's proof generalizes to the case $F$ is $(m+1)$-chromatic with a color-critical edge, showing $\ex_r(n,F^{+r})=|\cT_{m}^r(n)|$ for $n$ large enough. It implies the same for Berge hypergraphs, which is Corollary \ref{criti}. 

Another corollary of Proposition \ref{edgecrit} is that any $(k+1)$-tuple of cliques where the order of one of the cliques is larger than the order of any other clique by at least three, is color-resistant for large enough $n$. In the next section we show that the same holds for every $n$, without the restriction on the order of the cliques.

\section{The colored variant}

\begin{theorem} Any $(k+1)$-tuple of cliques is color-resistant for every $n$, i.e. if $H_1,\dots,H_k$ are cliques, then $\cN{{\mathrm col}}(H_1,\dots,H_k;K_m)$ is maximized by a monochromatic $T_{m-1}(n)$.

\end{theorem}

Note that it depends on the parameters which color gives the maximum, but obviously if $n$ is large enough, then it is the color $i$ such that $H_i$ is the largest clique.

The case $k=2$, $H_1=K_r$ and $H_2=K_2$ was proved in \cite{gmp}. Gerbner, Nagy, Patk\'os and Vizer \cite{gnpv} considered a variant, where in a $K_m$-free graph $G$ we count the blue copies of $K_r$, add $t$ times the red edges, and subtract $t-1$ times all the edges of $G$.
They proved that again a monochromatic $T_{m-1}(n)$ attains the maximum. Both of those proofs use Zykov's symmetrization method \cite{zykov} in a straightforward but involved way. We follow the steps from \cite{gnpv} in the first half of the proof. 

\begin{proof}
Let $G$ be a $K_m$-free graph with the largest value of $x(G):=\cN{{\mathrm col}}(H_1,\dots,H_k;G)$, where every $H_i$ is a clique. Among such graphs, we pick one with the smallest number of colors. For a vertex $v$, we let $d_i(v,G)$ denote the number of copies of $H_i$ in $G$ containing $v$, and let $d^*(v):=d^*(v,G):=\sum_{i=1}^k d_i(v,G)$. 

For two vertices $u$ and $v$, we say that we \textit{symmetrize} $u$ to $v$ if we delete all the
edges incident to $u$, and then for every edge $vw$, we add the edge $uw$ of the same color. We will apply this to non-adjacent vertices. It is well-known and easy to see that no $K_m$ is created this way. It is also easy to see that if $d^*(u)\le d^*(v)$, then $x(G)$ does not decrease, while if $d^*(u)<d^*(v)$, then $x(G)$ increases, which is a contradiction. Which means that when we apply such symmetrization steps, unconnected vertices always have the same $d^*$-value. 

We will change the graph with symmetrization steps to other graphs $G'$, but with an abuse of notation, we will use the same notation $d^*(v)$ for $d^*(v,G')$. It should not cause confusion, as we always deal with one graph at a time.

We will apply several symmetrization steps in the next part of the proof. In the first phase, we pick a vertex $v$.
Recall that each vertex not connected to $v$ has the same $d^*$-value. Then one by one we symmetrize to $v$ every vertex that is not connected to it. 

After this, we
obtain an independent set $A$ of vertices such that each vertex $w\not\in A$ is connected to each vertex
$u\in A$, by edges of the same color. Observe
that this property does not change in further symmetrization steps.

In the second phase we pick a vertex not in $A$ and do the same what we did in the first phase.
This way we obtain another independent set, and so on. After at most $m-1$ phases we obtain a complete multipartite graph $G'$ with at most $m-1$ classes such
that for any two of its classes, all the edges between them are of the same color, and the vertices inside a class have the same $d^*$-value. Obviously $x(G')=x(G)$. If $G$ has at least $m-1$ vertices, then we can assume that $G'$ has $m-1$ parts, as otherwise we can add edges to it, increasing the number of parts without decreasing $x(G')$. 

Let us introduce a symmetrization operation on classes of $G'$. When we symmetrize $A$ to $B$, for every third class $C$, we recolor the edges between $A$ and $C$ with the colors of the edges between $B$ and $C$. 
We mimic the previous part of the proof now, with classes of $G'$ playing the role of the vertices. However, any two classes are connected, so one of the colors will play the role of the non-edges. 
We pick the color such that the smallest clique, say $K_p$ we count is of that color. If there are no edges of that color, we pick the second smallest clique and so on. Let blue be the first color in this ordering such that there exist blue edges in $G'$ and we always symmetrize two classes connected by blue edges. 

We claim that either symmetrizing $A$ to $B$, or symmetrizing $B$ to $A$ does not decrease $x$. Consider the contribution of part $A$ to $x$, i.e., the sum of the number of copies of $H_i$ in color $i$ that contain a vertex from $A$. Let $y$ be this number without counting the blue cliques that contain vertices from both $A$ and $B$, and analogously $z$ denote the contribution of $B$ without those blue cliques. Observe that the blue cliques that contain vertices from both $A$ and $B$ remain in the graph after the symmetrization. Therefore, $x$ increases by $y-z$ or $z-y$, depending on the direction of the symmetrization, hence indeed at least one of the two symmetrizations does not decrease $x$. 

Applying such symmetrization steps repeatedly,
we obtain first a set of classes each connected by blue edges, then another set, and so on, i.e. at the end of this process, being either connected by blue edges or not connected is an equivalence relation on the vertices. Let $G''$ be the graph obtained this way.

Let $\cA$ denote an equivalence class of this relation. Then $\cA$ induces a complete multipartite graph itself. Then every vertex of $\cA$ and every vertex of another class $\cB$ are connected, with edges of the same color. 
Then we pick an arbitrary other color which actually appears in $G''$, say red, which corresponds to cliques $K_q$ (thus $q>p$). We introduce a symmetrization operation on equivalence classes, exactly the same way as earlier. For two classes $\cA$ and $\cB$ connected by red edges, and every third class $\cC$, we recolor all the edges from $\cC$ to $\cB$ to the color of the edges from $\cC$ to $\cA$ or the other way around.

Again, at the end of this process, in the resulting graph $G'''$, being unconnected or connected by red or blue edges is an equivalence relation. This means that the blue $p$-cliques are inside the equivalence classes of $G''$, and several such equivalence classes are connected by red edges, that is where we can find red $q$-cliques. Now we have three kind of classes, so we name them to help distinguish. We call the partite sets of $G'$ a \textit{small pack}, these are independent sets in each of $G'$, $G''$ and $G'''$. We call the equivalence classes of $G''$ \textit{medium packs}, these induce blue monochromatic complete multipartite graphs in $G''$ and $G'''$, where each partite set is a small pack. Finally, we call the equivalence classes of $G'''$ \textit{large packs}. These are red-blue complete multipartite graphs, where each partite set is a small pack. Large packs consist of multiple medium packs, where two vertices from different medium packs are connected by a red edge.
We remark that we could continue this procedure with other colors and obtain even larger packs, but this will suffice for us.

As there are red edges in $G'''$, there is a medium pack that is incident to red edges. If there are less than $q-1$ small packs incident to that medium pack by red edges (these are the small packs in the same large pack, but in a different medium pack), then we could recolor these red edges to blue without deleting any red $K_q$, thus without decreasing $x(G)$. Then we repeat this to the other red edges. As deleting a color would contradict our assumptions, we can find a medium pack $\cA$ such that at least $q-1$ small packs are incident to $\cA$ by red edges. Let $\cF$ denote the family of these small packs.

Let us recolor all the edges inside $\cA$ to red for every $\cA\in \cF$. Then $x(G)$ decreases by the number of blue $p$-cliques inside that medium pack, but increases by some red $q$-cliques. We need that the number of new red $q$-cliques is larger than the number of blue $p$-cliques inside $\cA$. But it is trivial, as for each blue $p$-clique inside $\cA$, we can find a new red $q$-clique by picking a vertex from each of $q-p$ small packs in $\cF$ (and for distinct $p$-cliques, we obtain distinct $q$-cliques this way). As $q-1>q-p>0$, there are more than one ways to pick the small packs, finishing the proof (even without considering that small packs may have more than one vertices).
\end{proof}




\smallskip

Let us continue with some examples for tuples that are not color-resistant. First we show an example for infinitely many $n$. Chase \cite{chase} showed that for the star $S_\ell$ with $\ell$ vertices, we have $\ex(n,K_r,S_\ell)=\cN(K_r,G)$, where $G$ consists of $\lfloor n/(\ell-1)\rfloor$ vertex disjoint copies of $K_{\ell-1}$, and a clique on the remaining vertices. Let $n=6p+2$ and consider $\cex(n,(K_3,K_2),S_7)$. Then it is obvious that taking $p$ copies of blue $K_6$, the remaining single edge should be red. This construction is clearly better than taking the same graph with only red edges and better than taking only blue edges (as the largest degree is at most 5, that would mean at most $5n/2$ blue edges). Therefore, $(K_3,K_2,S_7)$ is not color-resistant.



Let us show an example where $\cex(n,(H_1,H_2),F)>\max\{\ex(n,H_1,F),\ex(n,H_2,F)\}$ holds for every $n$ large enough.
Let $F_2$ be the \textit{2-fan}, which consists of two triangles sharing a vertex.
Erd\H os, F\"uredi, Gould, and Gunderson \cite{efgg} showed $\ex(n,F_2)=\lfloor n^2/4\rfloor+1$, where the construction is $T_2(n)$ with an arbitrary edge added. Moreover, this is the only $F_2$-free graph with that many edges. Gerbner and Palmer \cite{gp3} showed $\ex(n,C_4,F_2)=\cN(C_4,T_2(n))=|E(T_2(n))||E(T_2(n-2))|=\frac{1}{4}\lfloor n/2\rfloor\lceil n/2\rceil(\lfloor n/2\rfloor-1)(\lceil n/2\rceil-1)$.

\begin{proposition}
If $n$ is sufficiently large, then we have $\cex(n,(C_4,K_2),F_2)=|E(T_2(n))||E(T_2(n-2))|+1$.
\end{proposition}

\begin{proof}
For the lower bound, one can take a blue $T_2(n)$ and add an arbitrary edge in red. 

For the upper bound, we follow the proof of the bound on $\ex(n,C_4,F_2)$ by Gerbner and Palmer \cite{gp3}. If $G$ is an $F_2$-free graph on $n$ vertices, and it has $\lfloor n^2/4\rfloor +1$ edges, we are done by the uniqueness of the extremal graph in the result of Erd\H os, F\"uredi, Gould, and Gunderson \cite{efgg}. If $|E(G)|\le\lfloor n^2/4\rfloor$ edges, Gerbner and Palmer showed that every edge is in at most $|E(T_2(n-2))|$ copies of $C_4$. 

If we have $x$ red edges, then $\cN^{{\mathrm col}}(C_4,K_2;G)\le(|E(G)|-x)|E(T_2(n-2))|+x\le |E(G)||E(T_2(n-2))|$, completing the proof.

\end{proof}

Both the above examples were built on the same principle: we take a graph $G$ that is extremal for $\ex(n,H,F)$ and has some edges not contained in any copy of $H$. Then we can take a blue $G$ and recolor those edges to red. In the above examples, the two-colored construction was larger by one than the monocolored one, but we could easily modify the first example to obtain a larger constant difference. 

Let $S_r$ denote the star on $r$ vertices. We denote by $F(n)$ the graph we obtain from $S_n$ by adding a matching of size $\lfloor (n-1)/2\rfloor$ on the leaves of the star (thus $F_2=F(5)$). Let $F(n)^*$ be the blue-red graph with the same set of edges, where the edges of the star are blue, and the edges of the matching are red.
Gerbner \cite{gerbner2} showed that if $n$ is large enough, then for $r\ge 4$ we have $ex(n,S_r,C_4)=\binom{n-1}{r-1}=\cN(S_r,S_n)$, while $\ex(n,S_3,C_4)=\cN(S_3,F(n))$. In the case $r\ge 4$, $F(n)^*$ shows that $\cex(n,(S_r,K_2),C_4)\ge \max \{\ex(n,S_r,C_4),\ex(n,K_2,C_4)\}+\lfloor (n-1)/2\rfloor$ if $n$ is large enough. This shows the difference can be linear, but we can do better by taking red matchings instead of red edges.


Let $M_t$ denote the matching with $t$ edges. It was shown in \cite{gerbner2} that $\ex(n,M_t,C_4)=(1+o(1))\ex(n,C_4)^t/t!=\Theta(n^{3t/2})$. Let us describe the simple argument here, as we will use it below. In fact, any graph with $f(n)=\omega(n)$ edges contains $(1+o(1))f(n)^t/t!$ copies of $M_t$. The upper bound follows from the fact that we pick $t$ edges, each at most $f(n)$ ways, and we count each copy of $M_t$ exactly $t!$ times. For the lower bound, we again pick the edges one by one. Observe that each time we can pick at least $f(n)-2(t-1)n=(1-o(1))f(n)$ edges, as we only have to subtract those that are incident to a previously picked edge.

Thus if $t<2(r-1)/3$ and $n$ is large enough, then $\ex(n,S_r,C_4)\}\ge\ex(n,M_t,C_4)$. Therefore, $F(n)^*$ shows $\cex(n,(S_r,M_t),C_4)\ge \max \{\ex(n,S_r,C_4),\ex(n,M_t,C_4)\}+ \binom{\lfloor(n-1)/2\rfloor}{t}$ if $2\le t<2(r-1)/3$ and $n$ is large enough, hence
the difference can be above any polynomial.

Let us examine this example in more detail. Observe that determining $\cex(n,(S_r,M_t),C_4)$ completely for every $r$ and $t$ and large enough $n$ would include the case $r=2$, $t=1$, which is $\ex(n,C_4)$. Despite significant effort by many researchers, this problems is still unsolved. However, in every case we can either determine $\cex(n,(S_r,M_t),C_4)$, or show that a red monochromatic graph gives the maximum.

\begin{theorem} We have
\begin{displaymath}
\cex(n,(S_r,M_t),C_4)=
\left\{ \begin{array}{l l}
\cN(S_3,F(n)) & \textrm{if\/ $r=3$ and $t=1$},\\
\cN^{{\mathrm col}}(S_r,M_t;F(n)^*) & \textrm{if\/ $r\ge 4$, $t<2(r-1)/3$ and $n$ is large enough},\\
\cN^{{\mathrm col}}(S_r,M_t;F(n)^*) & \textrm{if\/  $r=4$, $t=2$ and $n$ is large enough},\\
\ex(n,M_t,C_4) & \textrm{for other values of $t$ and $r$, if $n$ is large enough}.
\\
\end{array}
\right.
\end{displaymath}
\end{theorem}

\begin{proof} Let $G$ be an $n$-vertex blue-red $C_4$-free graph which contains the most blue $S_r$ plus red $M_t$. Assume first that $r=3$ and $t=1$. In this case we use induction on $n$, the base cases $n= 3$ and $n=4$ are trivial. Let us asume $n\ge 5$. If there is a red edge $uv$, we can recolor it to blue, decreasing the number of red edges by one, and increasing the number of blue $S_3$'s by at least one, unless both $u$ and $v$ have no blue edge incident to them. Moreover, if there is a red cycle, we can recolor its edges, and again the number of blue $S_3$'s increases by at least the number of deleted red edges. 

Therefore, after executing this recoloring for every red edge and red cycle, each connected component is monochromatic in the resulting graph $G'$, and the red components are trees. By replacing a red component on $m\ge 3$ vertices by a blue $F(m)$, we delete $m-1$ red edges and add more than $m-1$ blue $S_3$'s, a contradiction. If there are $\ell>1$ red edges, each is a component, then we replace them with a blue $F(2\ell)$ to obtain a contradiction. Finally, if there is a single red edge, then
$\cN^{{\mathrm col}}(S_3,M_1;G)\le 1+\cex(n-2,(S_3,M_1),C_4)=1+\binom{n-3}{2}+2\lfloor (n-3)/2\rfloor$ by the induction hypothesis, which finishes the proof.

\smallskip

Assume now that $r\ge 4$ and $t<2(r-1)/3$. Let $\Delta$ be the largest blue degree in $G$, i.e. the largest $d_{blue}(v)$, where $d_{blue}(v)$ is the number of blue edges incident to $v$. We count the blue stars the following way. We pick two vertices $u$ and $v$, $\binom{n}{2}$ ways. They have at most one common neighbor $w$. We count the blue copies of $S_r$ containing $u$ and $v$ as leaves. There are at most $\binom{d_{blue}(w)-2}{r-3}\le\binom{\Delta-2}{r-3}$ such copies. If we count these for every pair of vertices, we count every blue copy of $S_r$ exactly $\binom{r-1}{2}$ times. Therefore, there are at most $\binom{n}{2}\binom{\Delta-2}{r-3}/\binom{r-1}{2}$ blue copies of $S_r$, and $O(n^{3t/2})=o(n^{r-1})$ red copies of $M_t$ in $G$. If $\Delta\le n-3$, then the sum of these two quantities is less than $\binom{n-1}{r-1}$, finishing the proof in this case. 

If $\Delta=n-2$, then let $x$ be a vertex of degree $n-2$ and $y$ be the only vertex not joined to $x$ by a blue edge. As $x$ and $y$ have at most one common neighbor $z$, we know that $y$ has degree at most 2. Observe that $z$ has degree at most three, and vertices different from $x,y,z$ have degree at most two. This shows that the number of blue $S_r$'s is at most $\binom{n-2}{r-1}+1$, where we have the plus 1 only if $r=4$ and $yz$ is a blue edge. Observe that we have at most $\lfloor (n-2)/2\rfloor$ edges between the neighbors of $x$, and those plus potentially $yz$ and $xy$ are the only red edges. Therefore, $\cN^{{\mathrm col}}(S_r,M_t;G) \le \binom{n-2}{r-1}+1+\binom{n/2+1}{t}<\binom{n-1}{r-1}+\binom{\lfloor (n-1)/2\rfloor}{t}$, finishing the proof in this case.

If $\Delta=n-1$ and $u$ has degree $n-1$, then there can only be independent edges in its neighborhood. Those edges cannot appear in any blue $S_r$, thus we have at most $\binom{n-1}{r-1}$ blue $S_r$'s, and the additional matching can contain at most $\binom{\lfloor (n-1)/2\rfloor}{t}$ red $M_t$'s.

\smallskip

Assume now $r=4$ and $t=2$. The same calculation as above shows that there are at most $\binom{n}{2}(\Delta-2)/3$ blue copies of $S_4$. 
Let $x$ be a vertex of degree $\Delta$. Let $a=n-\Delta-1$, and $A$ be the set of $a$ vertices not connected to and different from $x$. If $a=0$, then we are done. Observe that every other vertex is connected to at most one neighbor of $x$, thus there are $O(n)$ edges incident to $x$ or its neighbors. This shows that $|E(G)|\le (1+o(1))\frac{1}{2}(n-\Delta)^{3/2}+O(n)$, thus there are at most $|E(G)|^2/2\le (1+o(1))(n-\Delta)^3/8+O(n^2)$ red copies of $M_2$. If $a=\Omega(n)$, then we have $\cN^{{\mathrm col}}(S_4,M_2;G)\le \binom{n}{2}(\Delta-2)/3+(1+o(1))(n-\Delta)^3/8+O(n^2)<\binom{n-1}{3}\le \cN^{{\mathrm col}}(S_4,M_2;F(n)^*)$, a contradiction. 

Hence we can assume that $a=o(n)$. Observe that every vertex $y\neq x$ has degree at most $a+2$. Indeed, $y$ is connected to at most one neighbor of $x$. The $a$ vertices in $A$ are incident to $O(a^{3/2})$ edges, as each of them is connected to at most one neighbor of $x$, and there are $O(a^{3/2})$ edges inside $A$ because of the $C_4$-free property. There are $O(n+a^{3/2})$ red edges altogether, thus there are $O(a^{3/2}n+a^3)=o(an^2)$ red copies of $M_2$ containing a vertex from $A$. We claim that there are $o(an^2)$ blue copies of $S_4$ containing a vertex from $A$. 
Indeed, the number of copies totally inside $A$ is $O(a^3)$. Otherwise we have to pick one of the $a$ edges connecting $A$ to $V(G)\setminus A$, and then two more neighbors of one of the endpoints of that edge.

Let us delete all the edges incident to $A$. By the above, we deleted $o(an^2)$ red copies of $M_2$ and blue copies of $S_4$. Then we connect each vertex of $A$ to $x$ by a blue edge, creating $\Omega(an^2)$ new blue copies of $S_4$. As the resulting graph is $C_4$-free, this is a contradiction.

\medskip

Assume now that $t\ge 2(r-1)/3$. Observe that in this case $\ex(n,M_t,C_4)=(1+o(1))\frac{n^{3t/2}}{t!2^t}$ and $\ex(n,S_r,C_4)=(1+o(1))\frac{n^{r-1}}{(r-1)!}$. In the case $t>2(r-1)/3$, the first has a larger order of magnitude. In the case $t=2(r-1)/3$ and $r>4$, they have the same order of magnitude, but the constant factor is larger for the first one. Moreover, recall that $G$ contains at most $|E(G)|^t$ copies of $M_t$. If $G$ has $o(n^{3/2})$ red edges, then we are done, since it has $o(\ex(n,M_t,C_4))$ red copies of $M_t$, and the number of copies of $S_r$'s is less than $(1-c)\frac{n^{3t/2}}{t!2^t}$ for some constant $c>0$. So we may assume that $G$ has $o(n^{3/2})$ red edges from now.

We will show that $G$ is red monochromatic.
Assume indirectly that $G$ contains a blue edge $uv$.  First we show that there is a blue star of size $\Omega(n^{3/4})$ and $t$ is close to $2(r-1)/3$. Observe that $uv$ is in at most $\binom{d(u)-1}{r-2}+\binom{d(v)-1}{r-2}$ (blue) copies of $S_r$. On the other hand, there are $\Theta(n^{3/2})$ red edges in $G$, thus $\Theta(n^{3/2})$ red edges in $G$ are independent from $uv$, hence there are $\Theta(n^{3(t-1)/2})$ red copies of $M_{t-1}$ in $G$ that extend to an $M_t$ with $uv$. Thus, recoloring $uv$ to red increases $\cN^{{\mathrm col}}(S_r,M_t;G)$ in the case $\binom{d(u)-1}{r-2}+\binom{d(v)-1}{r-2}=o(n^{3(t-1)/2})$, in particular if $r-2<3(t-1)/2$ or if both $d(u)$ and $d(v)$ are $o(n^{\frac{3t-3}{2r-4}})$.
This shows that in fact every blue edge must be contained in a blue star $S_q$ with $q=\Omega(n^{\frac{3t-3}{2r-4}})=\Omega(n^{\frac{2r-5}{2r-4}})=\Omega(n^{3/4})$. 

Now we will show that there is a star of size $\Theta(n)$ and $t= 2(r-1)/3$.
Let $u_1,\dots,u_\ell$ be the centers of blue stars with at least $q$ vertices. We claim that $\ell=O(n^{1/4})$. Indeed, if we go through those centers in an arbitrary order, the first star has at least $q$ vertices, the second star contains at least $q-2$ vertices not in the first star, the third star contains at least $q-4$ vertices not in the first two stars, and so on. This shows the number of vertices in those stars is at least $q\ell-\sum_{i=1}^\ell 2i$, but at most $n$, which proves the claimed upper bound.

Let $q_i$ be the order of the blue star with center $u_i$, i.e. $q_i=d_{blue}(u_i)+1$. Without loss of generality, let $q_1\ge q_i$ for every $i$, and let $v_1,\dots, v_{q_1-1}$ be the vertices connected to $u_1$ by a blue edge. Observe that $v_i$ and $v_j$ does not have any common neighbor besides $u_1$, and $v_i$ has at most one neighbor among the $v_j$s, thus $\sum_{i=1}^{q_1-1} d(v_i)\le n+q_1-2$. The $q_1-1$ edges of this $S_{q_1}$ 
are contained altogether in at most $\binom{q_1-1}{r-1}+\sum_{i=1}^{q_1-1} \binom{d_{blue}(v_i)-1}{r-2}$ blue copies of $S_r$. Given that $\sum_{i=1}^{q_1-1} (d_{blue}(v_i)-1)\le n$ and $d_{blue}(v_i)\le q_1$, it is easy to see that $\sum_{i=1}^{q_1-1} \binom{d_{blue}(v_i)-1}{r-2}\le \frac{n}{q_1-1}\binom{q_1-1}{r-2}=O(nq_1^{r-3})=o(q_1^{r-1})$, where we use $q_1=\Omega(n^{3/4})$, hence there are  at most $(1+o(1))\binom{q_1-1}{r-1}$ blue copies of $S_r$ containing at least one of those edges. 

Let us now delete all the edges incident to any of $u_1,v_1,\dots,v_{q_1-1}$ to obtain $G'$. By the above, we deleted at most $(1+o(1))\binom{q_1-1}{r-1}$ blue copies of $S_r$.  On the other hand, we deleted at most $n+q_1$ edges, thus $O(n^{\frac{3(t-1)}{2}+1})$ red copies of $M_t$. Now we add a red $C_4$-free graph $G''$ with $\ex(q_1,C_4)=(1+o(1))q_1^{3/2}/2$ edges on these $q_1$ vertices $u_1,v_1,\dots,v_{q_1-1}$. The resulting graph $G'''$ is obviously $C_4$-free, since it consists of two $C_4$-free components. Let us consider the red copies of $M_t$ that are in $G'''$ but not in $G$. We can pick an edge from $G''$ and a red $M_{t-1}$ from $G'$. There are $\Theta(q_1^{3/2}n^{3(t-1)/2})$ ways to do this. Indeed, we have shown $G$ has $\Theta(n^{3/2})$ red edges, and then so does $G'$, as we deleted $O(n)$ red edges. Observe that $q_1^{3/2}=\Omega(n^{9/8})$, thus the number of deleted red copies of $M_t$ is $o(q_1^{3/2}n^{3(t-1)/2})$. The number of deleted blue copies of $S_r$ is $(1+o(1))\binom{q_1-1}{r-1}=O(q_1^{3/2}q_1^{3(t-1)/2})$. This shows that if $q_1=o(n)$ or $t>2(r-1)/3$, then we added more red copies of $M_t$ then the number of deleted red copies of $M_t$ and blue copies of $S_r$, a contradiction.

Let us now assume $q_1=\Theta(n)$ and $t=2(r-1)/3$ and delete all the vertices incident to any blue edge, to obtain $G_1$. We do it by going through the stars with centers $u_i$ as $i$ increases. For each star, any vertex is incident to at most one of its leaves, thus we delete $O(n)$ edges each time, thus altogether $O(n^{5/4})$ edges. This shows that we deleted $o(n^{3t/2})$ red copies of $M_t$. Let $q'$ be the number of vertices deleted and $q''=\ell+\sum_{i=1}^\ell (q_i-1)$. Then $q'\le q''\le q'+\binom{\ell}{2}$, since any two of the $\ell$ stars $S_{q_i}$ share at most one leaf. Let us consider the blue stars deleted. There are $\sum_{i=1}^\ell \binom{(q_i-1)}{r-1}\le \binom{q''-1}{r-1}$ copies with center $u_i$ for some $i$. For every other blue star, its leaves are among the $u_i$'s, thus there are at most $n\binom{\ell}{r-1}=o(n^{3t/2})$ such copies.

This way we obtained 

\begin{align*}
  \cN^{{\mathrm col}}(S_r,M_t;G)\le (1+o(1))\binom{q''-1}{r-1}+\cN^{{\mathrm col}}(S_r,M_t;G_1)=(1+o(1))\binom{q'-1}{r-1}+\cN^{{\mathrm col}}(S_r,M_t;G_1) \\
   =(1+o(1))\binom{q'-1}{r-1}+\cN(M_t,G_1)\le (1+o(1))\binom{q'-1}{r-1}+\ex(n-q',M_t,C_4) \\
  =(1+o(1))\binom{q'-1}{r-1}+(1+o(1))\frac{(n-q')^{3t/2}}{t!2^t}=(1+o(1))(\frac{q'^{r-1}}{(r-1)!}+\frac{(n-q')^{r-1}}{t!2^t}).   
\end{align*}

As $q'=\Theta(n)$, this is asymptotically smaller than \[(1+o(1))(\frac{q'^{r-1}}{t!2^t}+\frac{(n-q')^{r-1}}{t!2^t})\le (1+o(1))\frac{n^{r-1}}{t!2^t}=\ex(n,M_t,C_4),\] a contradiction finishing the proof.
\end{proof}

Let us remark that in the case of two colors, there is a natural way to improve the trivial lower bound $\ex(n,H_i,F)$ on $\cex(n,(H_1,H_2),F)$. We take an $n$-vertex $F$-free graph with $\ex(n,H_1,F)$ copies of $H_1$, and consider the unused edges, those that are not contained in any copy of $H_1$. We color those edges red, and the other edges blue (and we can do the same for $H_2$). In the case of more colors, the same approach can also give an improvement, but it is not obvious how to color the unused edges.

All the examples above are of this type, thus one could think this lower bound might be always sharp. However, we can modify the first example to show that this is not the case. Consider $\cex(8p+5,(K_4,K_3),S_9)$. Then the extremal construction for both $\ex(8p+5,K_4,S_9)$ and $\ex(8p+5,K_3,S_9)$ consists of $p$ copies of $K_8$ and one copy of $K_5$, and there are no unused edges, thus the lower bound is given by a monochromatic (in fact, blue monochromatic) graph. On the other hand, it is obvious that the $K_5$ should be red, and the $K_8$'s should be blue to maximize the number of blue $K_4$'s and red $K_3$'s.


\section{The uncolored variant}

First we show that an analogue of Conjecture \ref{mainconj} does not hold in the uncolored case. More precisely, we show examples such that for none of the $F$-free graphs $G$ with $\cN(H_i,G)=(1+o(1))\ex(n,H_i,F)$ for some $i$ have $\cN(H_1,\dots,H_k;G)=(1+o(1))\ex(n,(H_1,\dots,H_k),F)$.

Let us consider $\ex(n,(K_{a,b},K_{s,t}),K_3)$. As we have mentioned, this is attained by a complete bipartite graph $G$ due to the result of Schelp and Thomason, but not necessarily a balanced one. In fact, if $a=b$, $s=1$ and $t=2a-1$, then the complete bipartite graph $G'$ with the most copies of $K_{a,b}$ is balanced, while the complete bipartite graph $G''$ with the most copies of $K_{s,t}$ is very unbalanced. It is not surprising that $G$ must be between $G'$ and $G''$. It was observed by Brown and Sidorenko \cite{brosid} that the maximum number of copies of $K_{s,t}$ in bipartite graphs is obtained in $K_{m,n-m}$ with $m=(1+o(1))p$, where $p$ is the maximum of $x^s(1-x)^t+x^t(1-x)^s$ on $[0,1]$. It is not hard to see that to count both $K_{a,b}$ and $K_{s,t}$, we have to maximize $x^s(1-x)^t+x^t(1-x)^s+x^a(1-x)^b+x^b(1-x)^a$.  For $\ex(n,(K_{3,3},K_{1,5}),K_3)$ a simple calculation shows that indeed, $G$ is very far from both $G'$ and $G''$, the larger part is of order rougly $0.78n$ in $G$, $n/2$ in $G'$ and $0.83$ in $G''$.

\bigskip

One of the main conjectures (Erd\H os \cite{erdos}) of generalized Tur\'an problems was that the largest number of pentagons among triangle-free graphs is in the balanced blow-up of the pentagon. It was proved in \cite{G2012,HaETAL}.
Here we study what happens if we count another graph as well. We pick some other five-vertex graphs, so that there can be $\Theta(n^5)$ copies of them in $K_3$-free graphs. Also, the extremal graph for many of them are very different from the blow-up of the pentagon. 

We will use a result of Gy\H ori, Pach and Simonovits \cite{gypl}. They showed $\ex(n,P_\ell,K_3)=\cN(P_\ell,T_2(n))$.
Let $M$ be the five-vertex graph consisting of two independent edges and an independent vertex, and $M'$ be the graph consisting of a $P_3$ and an independent edge. 

\begin{proposition}\label{ujjj}
$\ex(n,M,K_3)=\cN(M,T_2(n))$ and $\ex(n,M',K_3)=\cN(M',T_2(n))$.
\end{proposition}

\begin{proof}
When counting $M$, we pick an edge at most $|E(T_2(n))|$ ways, we pick another, independent edge at most $|E(T_2(n-2))|$ ways and then a fifth vertex at most $n-4$ ways. We have equality everywhere in the Tur\'an graph.
When counting $M'$, we pick an edge at most $|E(T_2(n))|$ ways, and then an independent copy of $P_3$ at most $\ex(n-2,P_3,K_3)=\cN(P_3,T_2(n-2))$ ways (using the result of Gy\H ori, Pach and Simonovits \cite{gypl} mentioned above). Again, we have equality everywhere in the Tur\'an graph.
\end{proof}

Let $C_4'$ be the graph obtained by joining a vertex to one of the vertices of a $C_4$.
We will also consider the path $P_5$. Observe first that the colored variant is trivial: every copy of $C_5$ contains five copies of $P_5$ and every copy of $P_5$ is counted at most once. Therefore, recoloring the edges of the color corresponding to the $C_5$ increases the total number, hence $(P_5,C_5,K_3)$ is color-resistant. We show that this 3-tuple is also resistant.

\begin{proposition}

$\ex(n,(P_5,C_5),K_3)=\cN(P_5,C_5;T_2(n))=\cN(P_5,T_2(n))$ and $\ex(n,(C_4',C_5),K_3)=\cN(C_4',C_5;T_2(n))=\cN(C_4',T_2(n))$.
\end{proposition}

\begin{proof}

Observe first that each of $P_5$, $C_4'$ and $C_5$ can be built such a way that we pick an $M$ and extend it by adding further edges. We will count how many ways we can add edges to a given copy of $M$ in order to obtain a copy of $H$. We will refer to this as building an $H$ from $M$. Analogously, they can be built from $M'$.

Simple case analysis shows that in a triangle-free graph $G$, there is at most one way to build a $C_5$ from $M$ and 
at most 4 ways to build a $C_4'$.

On the other hand, we counted 
every copy of $C_4'$ four times and every copy of $C_5$ 5 times.
In particular, in $T_2(n)$, for every $M$ we find 
4 copies of $C_4'$ and count them 4 times, thus $\cN(M,T_2(n))=\cN(C_4',T_2(n))$.

Let $x$ be the number of copies of $M$ that extend to $C_5$ in $G$, and $y$ be the number of other copies. Then $x+y$ is the number of copies of $M$ in $G$, thus $x+y\le \cN(M,T_2(n))$ using Proposition \ref{ujjj}. Therefore, the number of copies of $C_5$ and $C_4'$ in $G$ is at most $x/5+y\le x+y\le \cN(M,T_2(n))=\cN(C_4',T_2(n))$, completing the proof of the first statement. 

For the second statement, we use a similar argument but with $M'$ in place of $M$. Simple case analysis shows that in a triangle-free graph $G$, there is at most one way to build a $C_5$ from $M'$ and at most 2 ways to build a $P_5$. A copy of $P_5$ contains two copies of $M'$.

On the other hand, we counted 
every copy of $P_5$ 3 times and every copy of $C_5$ 5 times.
In particular, in $T_2(n)$, for every $M'$ we find 2 copies of $P_5$ and count every copy of $P_5$ two times, thus $\cN(M',T_2(n))=\cN(P_5,T_2(n))$.

Let $x'$ be the number of copies of $M'$ that extend to $C_5$ in $G$, and $y$ be the number of other copies. Then $x'+y'$ is the number of copies of $M'$ in $G$, thus $x'+y'\le \cN(M',T_2(n))$ using Proposition \ref{ujjj}. Therefore, the number of copies of $C_5$ and $C_4'$ in $G$ is at most $x'/5+2x'/3+y'\le x'+y'\le \cN(M',T_2(n))=\cN(C_4',T_2(n))$, completing the proof. 
\end{proof}

Let us continue with $K_{2,3}$. Gy\H ori, Pach and Simonovits \cite{gypl} showed $\ex(n,K_{2,3},K_3)=\cN(K_{2,3},T_2(n))$.
We have already mentioned a result of Andr\'asfai, Erd\H os and S\'os \cite{aes}, a special case of which states that a triangle-free graph that is not bipartite has a vertex of degree at most $2n/5$. We will also use a special case of a theorem of Brouwer \cite{br} that states  that in a triangle-free graph on at least 5 vertices, there are at most $|E(T_2(n-2))|-\lfloor n/2\rfloor+1$ edges.

We will use the following simple statement  multiple times.

\begin{lemma}\label{lemmik}
    Let $vv'$ be an edge of an $n$-vertex triangle-free graph $G$ and assume that by deleting $v$ and $v'$ from $G$ we obtain a bipartite graph $G'$. Then the minimum degree $d$ in $G$ is less than $(n+2)/3$.
\end{lemma}

\begin{proof}
    If $v$ is connected to a vertex $a\in A$, observe that all the other neighbors of $a$ are in $B$, and then $v$ cannot be connected to any of those at least $d-1$ vertices. If $v$ is also connected to a $b\in B$, then $v$ is not connected to the at least $d-1$ other neighbors of $b$ either. As those vertices are in $A$, it means the degree of $v$ is at most $n-1-(2d-2)=n-2d+1$, thus $d\le n-2d+1$ and we are done. 
    
    Thus $v$ cannot be connected to vertices both in $A$ and $B$, and the same holds for $v'$. That means $G$ is bipartite, unless $v$ and $v'$ both have neighbors in the same part, say $A$. As they do not have common neighbors, this implies that $|A|\ge 2d$, thus $|B|\le n-2d$. Vertices in $A$ are connected only to vertices in $B$ and at most one of $v$ and $v'$, thus have degree at most $|B|+1$, showing that $d\le n-2d+1$ and completing the proof.
\end{proof}

\begin{proposition}\label{k23}
$\ex(n,(K_{2,3},C_5),K_3)=\cN(K_{2,3},C_5;T_2(n))$.
\end{proposition}

\begin{proof} We apply induction on $n$. The base cases $n\le 5$ are trivial. The base case $n=6$ is simple: if a 6-vertex triangle-free graph contains $K_{2,3}$, then it cannot contain $C_5$, since there are no further edges among these vertices and the last vertex cannot be connected to vertices in both parts of the $K_{2,3}$. Therefore, $\ex(6,(K_{2,3},C_5),K_3)\le \max\{\ex(6,K_{2,3},K_3),\ex(6,C_5,K_3)\}=6$, since it is not hard to see that $\ex(6,C_5,K_3)=2$. Assume $n\ge 7$.

Let $G$ be a $K_3$-free graph. 
If $G$ is bipartite, we are done. Otherwise $G$ the minimum degree $d$ in $G$ is at most $2n/5$. Let us delete a vertex $v$ of degree $d$ and let $G'$ be the graph obtained this way. By induction $G'$ contains at most $\cN(K_{2,3},T_2(n-1))$ copies of $C_5$ or $K_{2,3}$.

Let $x$ denote the number of copies of $K_{2,3}$ and $C_5$ in $G$ containing $v$. Let $w$ be a vertex from the larger part of $T_2(n)$ and $y$ denote the number of copies of $K_{2,3}$ in $T_2(n)$ that contain $w$. 

\begin{claim}
$x\le y$
\end{claim}

\begin{proof}
We can count the copies of $C_5$ and $K_{2,3}$ containing $v$ the following way. We pick an edge incident to $v$ (at most $d$ ways), then an independent edge (at most $|E(T_2(n-2))|$ ways) and a fifth vertex ($n-4$ ways). Then on these five vertices there are at most one copy of $K_{2,3}$ or $C_5$, as any edge added to $K_{2,3}$ or $C_5$ would create a triangle.

Every $K_{2,3}$ where $v$ is in the larger part is counted four times this way, and every $K_{2,3}$ where $v$ is in the smaller part is counted six times, while every $C_5$ is counted four times. Thus we have that 

\begin{equation}\label{equ1}
    4x\le d(n-4)|E(T_2(n-2))|\le \frac{2n}{5}(n-4)|E(T_2(n-2))|.\end{equation}

Let us consider now $T_2(n)$ and $w$.
Let $y_1$ denote the number of those copies of $K_{2,3}$ where $w$ is in the smaller part, and $y_2=y-y_1$. We can count the copies of $K_{2,3}$ just like in $G$. First, for $y_1$, we pick an incident edge $\lfloor n/2\rfloor$ ways, an independent edge exactly $|E(T_2(n-2))|$ ways and a fifth vertex $\lfloor (n-4)/2\rfloor$ ways. Thus we have $6y_1\ge \lfloor n/2\rfloor\lfloor (n-4)/2\rfloor|E(T_2(n-2))|\ge \frac{(n-1)(n-5)}{4(n-4)}(n-4)|E(T_2(n-2))|$. Similarly, we have $4y_2\ge \lfloor \frac{n}{2}\rfloor \lceil(n-4)/2\rceil|E(T_2(n-2))|\ge \frac{(n-1)(n-5)}{4(n-4)}(n-4)|E(T_2(n-2))|$.


Combining these inequalities, we obtain that $y_1+y_2\ge 5(n-1)(n-5)|E(T(n-2))|/48$. Combining this
with (\ref{equ1}), we are done if $24n(n-4)\le 25(n-1)(n-5)$, which holds when $n\ge 52$.

For smaller values of $n$, we modify the proof by using Brouwer's theorem, mentioned above. Let $vv'$ be the first edge we picked and $G'$ be the graph we obtain by deleting $v$ and $v'$ from $G$. Assume first none of the graphs $G'$ we obtain this way is bipartite. 
Then $|E(G')|\le |E_T(n-2)|-\lfloor n/2\rfloor+1$ edges. Then we can improve (\ref{equ1}) to $4x\le \frac{2n}{5}(n-4)(|E(T(n-2))|-\lfloor n/2\rfloor+1)$. One can check that this is smaller than our lower bound on $y_1+y_2$ if $n\ge 7$. 

Assume now that $G'$ is bipartite. Then $d<(n+2)/3$ by Lemma \ref{lemmik}. That modifies  (\ref{equ1}) to $4x\le (n+1)(n-4)|E(T(n-2))|/3$. One can check that this is smaller than our lower bound on $y_1+y_2$ if $n\ge 5$, completing the proof.
\end{proof}
The number of copies of $K_{2,3}$ and $C_5$ in $G$ is at most $x+\cN(K_{2,3},T_2(n-1))\le y+\cN(K_{2,3},T_2(n-1))=\cN(K_{2,3},T_2(n))$, finishing the proof.
\end{proof}

A similar proof deals with $M$ and $M'$.

\begin{proposition} We have
$\ex(n,(M,C_5),K_3)=\cN(M,C_5;T_2(n))$ and $\ex(n,(M',C_5),K_3)=\cN(M',C_5;T_2(n))$.
\end{proposition}

\begin{proof}
Just as in the proof of Proposition \ref{k23}, we use induction on $n$, and the base cases $n\le 5$ are trivial. Let $G$ be an $n$-vertex triangle-free graph. If $G$ is bipartite, we are done by Proposition \ref{ujjj}.

Thus we can assume that in $G$ the minimum degree is $d\le 2n/5$, let $v$ be a vertex of degree $d$. Let $x$ denote the number of copies of $M$ where $v$ is one of the non-isolated vertices, plus the number of copies of $C_5$ containing $v$ in $G$. Let $x'$ denote the number of copies of $M'$ containing $v$ plus the number of copies of $C_5$ containing $v$ in $G$.
Let $w$ be a vertex of the larger part of $T_2(n)$ and $y$ be the number of copies of $M$ containing an edge incident to $w$ in $T_2(n)$. Let $y'$ denote the number of copies of $M'$ containing $w$ in $T_2(n)$.

\begin{claim}
We have $x\le y$ and $x'\le y'$. 
\end{claim}

\begin{proof}
First we show a simple argument that works in the case $n$ is even, and afterwards we show how to improve it for the missing case $n$ is odd. We pick $M$ such that $v$ is not the isolated vertex at most $d(n-4)|E(T_2(n-2))|$ ways by picking an edge incident to $v$, then an edge on the remaining $n-2$ vertices, and then a fifth vertex. There is at most one $C_5$ containing that $M$, and each $C_5$ containing $v$ is counted this way exactly four times. Thus we have that the number of copies of $C_5$ containing $v$ is at most $d(n-4)|E(T_2(n-2))|/4$, thus $x\le 5d(n-4)|E(T_2(n-2))|/4\le \frac{n}{2}(n-4)|E(T_2(n-2))|$. On the other hand, the same calculation in the Tur\'an graph yields  $y= \lfloor\frac{n}{2}\rfloor(n-4)|E(T_2(n-2))|$. This finishes the proof of the first statement if $n$ is even.

Similarly, we pick $M'$ by picking an edge incident to $v$, an independent edge, a fifth vertex, and finally connect the fifth vertex to an endpoint of one of the two edges picked earlier. There are at most two ways to pick that last edge because of the triangle-free property, thus we pick $M'$ at most $2d(n-4)|E(T_2(n-2))|$ ways. There is at most one $C_5$ containing that $M'$, and we count every $C_5$ five times. Thus we have that the number of copies of $C_5$ containing $v$ is at most $2d(n-4)|E(T_2(n-2))|/5$, thus $x'\le 12d(n-4)|E(T_2(n-2))|/5\le\frac{24n}{25}(n-4)|E(T_2(n-2))|$. On the other hand, the same calculation in the Tur\'an graph yields  $y'= 2\lfloor\frac{n}{2}\rfloor(n-4)|E(T_2(n-2))|$. This finishes the proof of the second statement if $n$ is even or $n\ge 25$.

Now we show how to improve the above bound on $x$. The improvement would work for even $n$, but for simplicity assume $n$ is odd. After we pick the first edge $vv'$ when picking a copy of $M$ or $M'$, let $G'$ be the graph on the remaining vertices and assume first $G'$ is not bipartite. Then the theorem of Brouwer \cite{br} we have mentioned earlier shows that $|E(G')|\le |E(T_2(n-2))|-\lfloor n/2\rfloor+1$. Using this in the calculation decreases the upper bound on $x$ by $\frac{n}{2}(n-4)(\lfloor n/2\rfloor-1)$, thus it becomes smaller than $y$. Similarly, the upper bound on $x'$ becomes smaller than $y'$.

Assume now $G'$ is bipartite and recall that $d$ is the smallest degree in $G$. 
Then $d<(n+2)/3$ by Lemma \ref{lemmik}. Therefore, 
the calculations in the first and second paragraph of this proof give the bound $x\le \frac{5}{4}\lfloor\frac{n+1}{3}\rfloor(n-4)|E(T_2(n-2))|$. We have $\frac{5}{4}\lfloor\frac{n+1}{3}\rfloor\le (n-1)/2$ if $n\ge 11$, and $\frac{5}{4}\lfloor \frac{n+1}{3}\rfloor\le (n-1)/2$ in the cases $n=7$ and $n=9$. Thus we have $x\le \lfloor\frac{n}{2}\rfloor(n-4)|E(T_2(n-2))|=y$.

Similarly, we have $x'\le 2\frac{n+1}{3}(n-4)|E(T_2(n-2))|\le 2\lfloor\frac{n}{2}\rfloor(n-4)|E(T_2(n-2))|=y'$ if $n\ge 5$, finishing the proof.
\end{proof}
The number of copies of $M$ and $C_5$ in $G$ is at most $x+\cN(M,T_2(n-1))\le y+\cN(M,T_2(n-1))=\cN(M,T_2(n))$, and similarly the number of copies of $M'$ and $C_5$ in $G$ is at most $x+\cN(M',T_2(n-1))\le y+\cN(M',T_2(n-1))=\cN(M',T_2(n))$, finishing the proof.
\end{proof}

\begin{corollary} 
Let $k\ge 2$ and $\cT$ be a $k$-tuple consisting of graphs $M$, $M'$, $C_4'$, $P_5$, $K_{2,3}$ and $C_5$. Then $\ex(n,\cT,K_3)=\cN(\cT;T_2(n))$.
\end{corollary}

\begin{proof}
We have proved the statement for $k=2$. Taking any $H$ of the graphs  $M$, $M'$, $C_4'$, $P_5$ together with $C_5$ shows that $T_2(n)$ contains the most copies of $H$ among triangle-free $n$-vertex graphs. Therefore, the case $k>2$ is also implied, as the copies of two elements of the $k$-tuple $\cT$ (including $C_5$, if it is in $\cT$) are maximized by $T_2(n)$, and any additional graph in the tuple is also maximized by $T_2(n)$.
\end{proof}

\end{document}